\documentclass[11pt,a4paper,leqno]{article}

\usepackage{enumerate}
\usepackage{amsfonts}
\usepackage{amsmath}
\usepackage{amsthm}
\usepackage{nicefrac}
\usepackage{textcomp}
\usepackage{xspace} 
\usepackage{amsthm}
\usepackage{units}
\usepackage[all,cmtip]{xy}
\usepackage{amssymb}
\usepackage{latexsym}

\newcommand{\C}{\mathbb{C}}
\newcommand{\N}{\mathbb{N}}

\newcommand{\mytilde}{{\raise.17ex\hbox{$\scriptstyle\mathtt{\sim}$}}\xspace}

\newtheorem{definition}{Definition}[section]
\newtheorem{theorem}[definition]{Theorem}
\newtheorem{corollary}[definition]{Corollary}

\newtheorem{lemma}[definition]{Lemma}
\theoremstyle{remark}

\title{Cowen-Douglas tuples and fiber dimensions}
\author{J\"{o}rg Eschmeier and Sebastian Langend\"{o}rfer}
\date{}

\begin{document}

\maketitle

\begin{center}
\parbox{11cm}{\small
\textbf{Abstract.} Let $T \in L(X)^n$ be a Cowen-Douglas system on a Banach space $X$. We use
functional representations of $T$
to associate with each $T$-invariant subspace $Y\subseteq X$ an integer called the fiber dimension
${\rm fd}(Y)$ of $Y$. 
Among other results we prove a limit formula for the fiber dimension, show that it is invariant under 
suitable changes of $Y$ and deduce
a dimension formula for pairs of homogeneous invariant subspaces of graded Cowen-Douglas tuples on Hilbert spaces.\\

}
\end{center}

\vspace{.8cm}

\section{Introduction}

Let $\Omega \subseteq \mathbb C^n$ be a domain and let $\mathcal H \subseteq \mathcal O(\Omega,\mathbb C^N)$ 
be a functional Hilbert space of $\mathbb C^N$-valued analytic functions on $\Omega$. The number
\[
{\rm fd} (\mathcal H) = {\rm max}_{\lambda \in \Omega} \dim \mathcal H_{\lambda},
\]
 where $ H_{\lambda} = \{f(\lambda); f \in \mathcal H\}$, is usually referred to as the fiber dimension of
 $\mathcal H$. Results going back to Cowen and Douglas \cite{CD}, Curto and Salinas \cite{CS} show that each
 Cowen-Douglas operator tuple $T \in L(H)^n$ on a Hilbert space $H$ is locally uniformly equivalent to the 
 tuple $M_z = (M_{z_1}, \ldots ,M_{z_n}) \in L(\mathcal H)^n$ of multiplication operators with the coordinate
 functions on a suitable analytic functional Hilbert space $\mathcal H$. In the present note we use corresponding
 model theorems for Cowen-Douglas operator tuples $T \in L(X)^n$ on Banach spaces to associate
 with each $T$-invariant subspace $Y \subseteq X$ an integer ${\rm fd}(Y)$ called
 the fiber dimension of $Y$. We thus extend results proved by L. Chen, G. Cheng and X. Fang in \cite{CCF}
 for single Cowen-Douglas operators on Hilbert spaces to the case of commuting operator systems on Banach
 spaces.\\
 
 By definition a commuting tuple $T = (T_1, \ldots ,T_n) \in L(X)^n$ of bounded operators on a Banach 
 space $X$ is a weak dual Cowen-Douglas tuple of rank $N \in \mathbb N$ on $\Omega$ if
 \[
\dim X/\sum_{i=1}^n (\lambda_i - T_i) X = N
\]
for each point $\lambda \in \Omega$. We call $T$ a dual Cowen-Douglas tuple if in addition
\[
\bigcap_{\lambda \in \Omega} \sum_{i=1}^n (\lambda_i - T_i) X = \{0\}.
\]
We show that weak dual Cowen-Douglas tuples $T \in L(X)^n$ admit local representations as multiplication
tuples $M_z \in L(\hat{X})^n$ on suitable functional Banach spaces $\hat{X}$ and prove that dual Cowen-Douglas
tuples can  be characterized as those commuting tuples $T \in L(X)^n$ that are locally jointly similar to a
multiplication tuple $M_z \in L(\hat{X})^n$ on a divisible holomorphic model space $\hat{X}$. We use the
functional representations of weak dual Cowen-Douglas tuples $T \in L(X)^n$ to associate with each linear
$T$-invariant subspace $Y \subseteq X$ an integer ${\rm fd}(Y)$ called the fiber dimension of $Y$.\\

Based on the observation that the fiber dimension ${\rm fd}(Y)$ of a closed $T$-invariant subspace 
$Y \in \; {\rm Lat}(T)$ is closely related to the Samuel multiplicity of the quotient tuple $S = T/Y \in
L(X/Y)^n$ on $\Omega$ we show that the fiber dimension of $Y \in {\rm Lat}(T)$ can be calculated by a limit formula
\[
{\rm fd}(Y) = n! \; \lim_{k \rightarrow \infty} \frac{\dim(Y+M_k(T-\lambda)/M_k(T-\lambda))}{k^n} \quad (\lambda \in \Omega),
\]
where $M_k(T-\lambda) = \sum_{|\alpha|=k} (T-\lambda)^{\alpha}X$. Furthermore, we show how to calculate the
fiber dimension using the sheaf model of $T$ on $\Omega$. We deduce that the fiber dimension is invariant against
suitable changes of $Y$ and we show that the fiber dimension for graded dual Cowen-Douglas tuples $T \in L(H)^n$ on
Hilbert spaces satisfies the dimension formula
\[
{\rm fd}(Y_1 \vee Y_2) +{\rm fd}(Y_1 \cap Y_2) = {\rm fd}(Y_1) + {\rm fd}(Y_2)
\]
for any pair of homogeneous invariant subspaces $Y_1, Y_2 \in {\rm Lat}(T)$. The proof is based on an idea from \cite{CF}
(see also \cite{CCF}) where a corresponding result is proved for analytic functional Hilbert spaces given by a complete
Nevanlinna-Pick kernel.

\section{Fiber dimension for invariant subspaces}

In the following, let $\Omega \subseteq \C^n$ be a domain, that is, a connected open set in $\C^n$. Let $D$ be a finite-dimensional vector space and let $M \subseteq \mathcal{O}(\Omega, D)$ be a $\C [z]$-submodule. 
We denote the point evaluations on $M$ by 
\begin{displaymath}
\epsilon_{\lambda}: M \rightarrow D, f \mapsto f(\lambda) \quad \quad (\lambda \in \Omega).
\end{displaymath}

For $\lambda \in \Omega$, the range of $\epsilon_{\lambda}$ is a linear subspace

\begin{displaymath}
M_{\lambda} = \{ f(\lambda); f \in M \} \subseteq D.
\end{displaymath}

\begin{definition}
\label{fibermod}
The number

\begin{displaymath}
{\rm fd}(M) = \max_{z \in \Omega} \dim M_z
\end{displaymath}

is called the fiber dimension of $M$. A point $z_0 \in \Omega$ with $\dim M_{z_0} = {\rm fd}(M)$ is called a maximal point of $M$.

\end{definition}

For any $\C[z]$-submodule $M \subseteq \mathcal{O}(\Omega, D)$ and any point $\lambda \in \Omega$, we have 
\begin{displaymath}
\sum_{i=1}^n (\lambda_i - M_{z_i})M \subseteq \ker \epsilon_{\lambda}.
\end{displaymath}
Under the condition that the codimension of $\sum_{i=1}^n (\lambda_i - M_{z_i})M$ is constant on $\Omega$, the question whether equality holds here is closely related to corresponding properties of the fiber dimension of $M$.

\begin{lemma}
\label{kernelpointev}

Consider a $\C[z]$-submodule $M \subseteq \mathcal{O}(\Omega, D)$ such that there is an integer $N$ with
\begin{displaymath}
\dim M / \sum_{i=1}^n (\lambda_i - M_{z_i})M  \equiv N
\end{displaymath}
for all $\lambda \in \Omega$. Then ${\rm fd}(M) \leq N$. If ${\rm fd}(M) < N$, then
\begin{displaymath}
\sum_{i=1}^n (\lambda_i - M_{z_i})M \subsetneq \ker \epsilon_{\lambda}
\end{displaymath}
for all $\lambda \in \Omega$. If ${\rm fd}(M) = N$, then there is a proper analytic set $A \subseteq \Omega$ with 
\begin{displaymath}
\Omega \backslash A \subseteq \{ \lambda \in \Omega; \dim M_{\lambda} = N \} = \{ \lambda \in \Omega; \sum_{i=1}^n (\lambda_i - M_{z_i})M = \ker \epsilon_{\lambda} \}.
\end{displaymath}

\end{lemma}

\begin{proof}
Since the maps
\begin{displaymath}
M / \sum_{i=1}^n (\lambda_i - M_{z_i})M \rightarrow M / \ker \epsilon_{\lambda} \cong {\rm Im}\, \epsilon_{\lambda}, [m] \mapsto [m]
\end{displaymath}
are surjective for $\lambda \in \Omega$, it follows that ${\rm fd}(M) \leq N$ and that
\begin{displaymath}
\{ \lambda \in \Omega; \dim M_{\lambda} = N \} = \{ \lambda \in \Omega; \sum_{i=1}^n (\lambda_i - M_{z_i})M = \ker \epsilon_{\lambda} \}.
\end{displaymath}
Hence, if ${\rm fd}(M) < N$, then $\sum_{i = 1}^n (\lambda_i - M_{z_i}) M \subsetneq \ker \epsilon_{\lambda}$ for all $\lambda \in \Omega$. A standard argument (cf. Lemma 1.4 in \cite{HSFT} and its proof) shows that there is a proper analytic set $A \subseteq \Omega$ such that 
\begin{displaymath}
\Omega \backslash A \subseteq \{ \lambda \in \Omega; \dim M_{\lambda} = {\rm fd}(M) \}.
\end{displaymath}
This observation completes the proof.
\end{proof}

In \cite{CCF} a fiber dimension was defined for invariant subspaces of dual Cowen-Douglas operators on Hilbert spaces. In the following we extend this definition to the case of weak dual Cowen-Douglas tuples on Banach spaces (see Definition \ref{CD}).

Let $X$ be a Banach space and let $T = (T_1,...,T_n) \in L(X)^n$ be a commuting tuple of bounded operators on $X$. For $z \in \C^n$, we use the notation $z - T$ both for the commuting tuple $z-T= (z_1 - T_1,...,z_n - T_n)$ and for the row operator
\begin{displaymath}
z - T: X^n \rightarrow X, (x_i)_{i=1}^n \mapsto \sum_{i=1}^n (z_i - T_i)x_i.
\end{displaymath}
With this notation, we have $\sum_{i=1}^n (z_i - T_i) X = \text{Im}(z - T)$. 

\begin{definition}
\label{CD}

Let $T \in L(X)^n$ be a commuting tuple of bounded operators on $X$ and let $\Omega \subseteq \C^n$ be a fixed domain.
We call $T$ a weak dual Cowen-Douglas tuple of rank $N \in \N$ on $\Omega$ if
\begin{displaymath}
\dim (X / \sum_{i=1}^n (z_i - T_i) X) = N
\end{displaymath}
for all $z \in \Omega$.
If in addition the condition
\begin{displaymath}
\bigcap_{z \in \Omega} {\rm Im}(z-T) = \{ 0 \}
\end{displaymath}
holds, then $T$ is called a dual Cowen-Douglas tuple of rank $N$ on $\Omega$.

\end{definition}

If $X = H$ is a Hilbert space, then a tuple $T \in L(H)^n$ is a dual Cowen-Douglas tuple on $\Omega$ if and only if the adjoint
$T^* = (T_1^*, \ldots , T_n^*)$ is a tuple of class $B_n(\Omega^*)$ on the complex conjugate domain $\Omega^* = \{\overline{z} ; z \in \Omega \}$
in the sense of Curto and Salinas \cite{CS}.
One can show (Theorem 4.12 in \cite{Wernet}) that, for a weak dual Cowen-Douglas tuple $T \in L(X)^n$ on a domain $\Omega \subseteq \C^n$, the identity
\begin{displaymath}
\bigcap_{z \in \Omega} \text{Im}(z - T) = \bigcap_{k=0}^{\infty} \sum_{\vert \alpha \vert = k} (\lambda - T)^{\alpha}X
\end{displaymath}
holds for every point $\lambda \in \Omega$. In particular, if $T \in L(X)^n$ is a dual Cowen-Douglas tuple on $\Omega$, then it is a dual Cowen-Douglas tuple on each smaller domain $\emptyset \neq \Omega_0 \subseteq \Omega$.

\begin{definition}
\label{HMS}
Let $\Omega \subseteq \C^n$ be open. A holomorphic model space of rank $N$ over $\Omega$ is a Banach space $\hat{X} \subseteq \mathcal{O}(\Omega, D)$ such that $D$ is an $N$-dimensional complex vector space and
\begin{enumerate}[{\rm (i)}]
\item $M_z \in L(\hat{X})^n$,
\item for each $\lambda \in \Omega$, the point evaluation $\epsilon_{\lambda} : \hat{X} \rightarrow D, \hat{x} \mapsto \hat{x}(\lambda)$, is continuous and surjective.
\end{enumerate}
A holomorphic model space $\hat{X}$ on $\Omega$ is called divisible if in addition, for $\hat{x} \in \hat{X}$ and $\lambda \in \Omega$ with $\hat{x}(\lambda) = 0$, there are functions $\hat{y}_1$,...,$\hat{y}_n \in \hat{X}$ with
\begin{displaymath}
\hat{x} = \sum_{i=1}^n (\lambda_i - M_{z_i})\hat{y_i}.
\end{displaymath}
\end{definition}

The multiplication tuple $M_z$ on a divisible holomorphic model space $\hat{X} \subseteq \mathcal{O}(\Omega, D)$ is 
easily seen to be a dual Cowen-Douglas tuple of rank $N = \text{dim} D$ on $\Omega$.

In the following let $T \in L(X)^n$ be a weak dual Cowen-Douglas tuple of rank $N$ on a fixed domain $\Omega \subseteq \C^n$. 
We extend a notion introduced in \cite{CCF} to our setting.

\begin{definition}
Let $\emptyset \neq \Omega_0 \subseteq \Omega$ be a connected open subset. A CF-representation of $T$ on $\Omega_0$ is a $\C[z]$-module homomorphism
\begin{displaymath}
\rho: X \rightarrow \mathcal{O}(\Omega_0, D)
\end{displaymath}
with a finite-dimensional complex vector space $D$ such that
\begin{enumerate}[{\rm (i)}]
\item $\ker \rho = \bigcap_{z \in \Omega} (z - T) X^n$,
\item the submodule $\hat{X} = \rho X \subseteq \mathcal{O}(\Omega_0, D)$ satisfies
\begin{displaymath}
{\rm fd}(\hat{X}) = \dim \hat{X} / \sum_{i=1}^n (\lambda_i - M_{z_i}) \hat{X}
\end{displaymath}
for all $\lambda \in \Omega_0$.
\end{enumerate}
\end{definition}

Let $\mathcal O(\Omega_0,D)$ be equipped with its canonical Fr\'{e}chet space topology.
Our first aim is to show that weak dual Cowen-Douglas tuples possess sufficiently many CF-representations
that are continuous and satisfy certain additional properties.

\begin{theorem}
\label{CFexists}
Let $T \in L(X)^n$ be a weak dual Cowen-Douglas tuple of rank $N$ on $\Omega$. Then, for each point $\lambda_0 \in \Omega$, there is a CF-representation $\rho: X \rightarrow \mathcal{O}(\Omega_0, D)$ of $T$ on a connected open neighbourhood $\Omega_0 \subseteq \Omega$ of $\lambda_0$ such that
\begin{enumerate}[{\rm (i)}]
\item $\rho: X \rightarrow \mathcal{O}(\Omega_0, D)$ is continuous,
\item $\hat{X} = \rho(X)$ equipped with the norm $\Vert \rho(X) \Vert = \Vert x + \ker \rho \Vert$ is a divisible holomorphic model space of rank $N$ on $\Omega_0$.
\end{enumerate}
\end{theorem}

\begin{proof}
Let $\lambda_0 \in \Omega$ be arbitrary. Choose a linear subspace $D \subseteq X$ such that 
\begin{displaymath}
X = (\lambda_0 - T)X^n \oplus D.
\end{displaymath}
Then $\dim D = N$. The analytically parametrized complex
\begin{displaymath}
T(z): X^n \oplus D \rightarrow X,\;  ((x_i)_{i=1}^n, y) \mapsto \sum_{i=1}^n (z_i - T_i)x_i + y
\end{displaymath}
of bounded operators between Banach spaces is onto at $z = \lambda_0$. By Lemma 2.1.5 in \cite{EP} there is an open polydisc $\Omega_0 \subseteq \Omega$ such that the induced map
\begin{displaymath}
\mathcal{O}(\Omega_0, X^n \oplus D) \rightarrow \mathcal{O}(\Omega_0, X), ((g_i)_{i=1}^n, h) \mapsto \sum_{i=1}^n (z_i - T_i)g_i + h 
\end{displaymath}
is onto. In particular, for each $z \in \Omega_0$, the linear map 
\begin{displaymath}
D \rightarrow X / \sum_{i=1}^n (z_i - T_i)X, x \mapsto [x]
\end{displaymath}
 is surjective between $N$-dimensional complex vector space. Hence these maps are isomorphisms. But then, for each $x \in X$ and $z \in \Omega_0$, there is a unique vector $x(z) \in D$ with $x - x(z) \in \sum_{i=1}^n (z_i - T_i) X$. By construction, for each $x \in X$, the mapping $\Omega_0 \rightarrow D, z \mapsto x(z)$, is analytic. The induced mapping 
 \begin{displaymath}
 \rho: X \rightarrow \mathcal{O}(\Omega_0, D), x \mapsto x(\cdot)
 \end{displaymath}
 is linear with
 \begin{displaymath}
 \ker \rho = \bigcap_{z \in \Omega_0} \sum_{i=1}^n (z_i - T_i)X = \bigcap_{z \in \Omega} \sum_{i = 1}^n (z_i - T_i)X.
 \end{displaymath}
For $x \in X$, $z \in \Omega_0$ and $j = 1,...,n$,
 \begin{displaymath}
 T_j x - z_j x(z) = T_j (x - x(z)) - (z_j - T_j)x(z) \in \sum_{i=1}^n (z_i - T_i)X.
 \end{displaymath}
Hence $\rho$ is a $\C[z]$-module homomorphism. Equipped with the norm $\Vert \rho(x) \Vert = \Vert x + \ker \rho \Vert$, the space $\hat{X} = \rho (X)$ is a Banach space and $M_z \in L(\hat{X})^n$ is a commuting tuple of bounded operators on $\hat{X}$. By definition
 \begin{displaymath}
 \rho(x) \equiv x \quad \quad \text{ for } x \in D.
 \end{displaymath}
 Hence the point evaluations $\epsilon_z: \hat{X} \rightarrow D$ ($z \in \Omega_0$) are surjective. Since the mappings
 \begin{displaymath}
 q_z: D \rightarrow X / \sum_{i=1}^n (z_i - T_i) X,\;  x \mapsto [x] \quad (z \in \Omega_0)
 \end{displaymath}
 are topological isomorphisms and since the compositions
 \begin{displaymath}
 X \rightarrow X / \sum_{i=1}^n (z_i - T_i) X,\; x \mapsto q_z(\epsilon_z(\rho(x))) = [x]
 \end{displaymath}
 are continuous, it follows that the point evaluations $\epsilon_z: \hat{X} \rightarrow D$ ($z \in \Omega_0$) are continuous. 
 Thus we have shown that $\hat{X} \subseteq \mathcal{O}(\Omega_0,D)$ with the norm induced by $\rho$ is a holomorphic model space. \\
 To see that $\hat{X}$ is divisible, fix a vector $x \in X$ and a point $\lambda \in \Omega_0$ such that $x(\lambda) = 0$. 
 Then there are vectors $x_1,...,x_n \in X$ with $x = \sum_{i=1}^n (\lambda_i - T_i)x_i$. Hence
 \begin{displaymath}
\rho(x) = \sum_{i=1}^n (\lambda_i - z_i) \rho(x_i) \in \sum_{i=1}^n (\lambda_i - M_{z_i}) \hat{X}.
\end{displaymath} 
To conclude the proof, it suffices to observe that
\begin{displaymath}
\dim (\hat{X} / \sum_{i=1}^n (\lambda_i - M_{z_i})\hat{X}) = \dim(\hat{X} / \ker \epsilon_{\lambda}) = \dim(\text{Im}\,\epsilon_{\lambda}) = \dim D = N
\end{displaymath}
for all $z \in \Omega_0$. 
\end{proof}

Note that, for a dual Cowen-Douglas tuple $T \in L(X)^n$ on a Banach space $X$, the mappings $\rho: X \rightarrow \hat{X} \subseteq \mathcal{O}(\Omega_0,D)$ 
constructed in the previous proof are isometric joint similarities between $T \in L(X)^n$ and the tuples $M_z \in L(\hat{X})^n$ on the divisible 
holomorphic model space $\hat{X} \subseteq \mathcal{O}(\Omega_0,D)$. 

\begin{corollary}
\label{quotient}
Let $T \in L(X)^n$ be a commuting tuple on a complex Banach space and let $\Omega \subseteq \C^n$ be a domain. The tuple $T$ is a dual Cowen-Douglas tuple of rank $N$ on $\Omega$ if and only if, for each point $\lambda \in \Omega$, there exist a connected open neighbourhood $\Omega_0 \subseteq \Omega$ of $\lambda$ and a joint similarity between $T$ and the multiplication tuple $M_z \in L(\hat{X})^n$ on a divisible holomorphic model space $\hat{X}$ of rank $N$ on $\Omega_0$.
\end{corollary}

\begin{proof}
The necessity of the stated condition follows from Theorem \ref{CFexists} and the subsequent remarks. Since the tuple $M_z \in L(\hat{X})^n$ on a divisible holomorphic model space of rank $N$ is a dual Cowen-Douglas tuple of rank $N$, and since the same is true for every tuple similar to $M_z \in L(\hat{X})^n$, also the sufficiency is clear.
\end{proof}

The preceding result should be compared with Corollary 4.39 in \cite{Wernet}, where a characterization of dual Cowen-Douglas tuple on suitable admissible domains in $\C^n$ is obtained.

There is a canonical way to associate with each weak dual Cowen-Douglas tuple of rank $N$ on $\Omega \subseteq \C^n$ a dual Cowen-Douglas tuple of rank $N$. 

\begin{corollary}
Let $T \in L(X)^n$ be a weak dual Cowen-Douglas tuple of rank $N$ on a domain $\Omega \subseteq \C^n$. Then the quotient tuple
\begin{displaymath}
T^{\text{CD}} = T / \bigcap_{z \in \Omega} \sum_{i=1}^n (z_i - T_i) X
\end{displaymath}
defines a dual Cowen-Douglas tuple of rank $N$ on $\Omega$.
\end{corollary}

\begin{proof}
Let $z_0 \in \Omega$ be arbitrary. Choose a CF-representation $\rho: X \rightarrow \mathcal{O}(\Omega_0, D)$ as in Theorem \ref{CFexists}. Then $\hat{X} = \rho(X) \subseteq \mathcal{O}(\Omega_0, D)$ is a divisible holomorphic model space of rank $N$ on $\Omega_0$. Since 
\begin{displaymath}
\ker \rho = \bigcap_{z \in \Omega} \sum_{i=1}^n (z_i - T_i) X,
\end{displaymath}
the mapping $\rho$ induces a similarity between $T^{\text{CD}}$ and $M_z \in L(\hat{X})^n$.
By Corollary \ref{quotient} the tuple $T^{\text{CD}}$ is a dual Cowen-Douglas tuple of rank $N$ on $\Omega$. 
\end{proof}

As before, let $T \in L(X)^n$ be a weak dual Cowen-Douglas tuple of rank $N$ on a domain $\Omega \subseteq \C^n$. We denote by $\text{Lat}(T)$ the set of closed subspaces $Y \subseteq X$ which are invariant under each component $T_i$ of $T$. Our next aim is to show that, for $Y \in \text{Lat}(T)$, the fiber dimension of $Y$ can be defined as 
\begin{displaymath}
\text{fd}(Y) = \text{fd}(\rho(Y)),
\end{displaymath}
where $\rho$ is an arbitrary CF-representation of $T$. We have of course to show that the number $\text{fd}(\rho(Y))$ is independent of the chosen CF-representation $\rho$. In the first step, we use an argument from \cite{CCF} to show that $\text{fd}(\rho_1(Y)) = \text{fd}(\rho_2(Y))$  for each pair of CF-representations $\rho_1$, $\rho_2$ over domains $\Omega_1, \Omega_2 \subseteq \Omega$ with non-trivial intersection.

\begin{lemma}
\label{invariancenontrivial}
Let $\Omega_1, \Omega_2 \subseteq \C^n$ be domains with $\Omega_1 \cap \Omega_2 \neq \emptyset$ and let $M_i \subseteq \mathcal{O}(\Omega_i, D_i)$ be $\C[z]$-submodules with finite-dimensional complex vector spaces $D_i$ such that 
\begin{displaymath}
{\rm fd}(M_i) = \dim M_i / (\lambda - M_z)M_i^n \quad (i=1,2, \lambda \in \Omega_i).
\end{displaymath}
Suppose that there is a $\C[z]$-module isomorphism $U: M_1 \rightarrow M_2$. Then, for any submodule $M \subseteq M_1$, we have
\begin{displaymath}
{\rm fd}(M) = {\rm fd} (UM).
\end{displaymath}
\end{lemma}

\begin{proof}
Using Lemma 1.4 in \cite{HSFT} and elementary properties of analytic sets, we can choose a proper analytic subset $A \subseteq \Omega_1 \cap \Omega_2$ such that each point $\lambda \in (\Omega_1 \cap \Omega_2) \backslash S$ is a maximal point for $M$, $M_1$ and $UM$. Fix such a point $\lambda$. If $f,g \in M$ are functions with $f(\lambda) = g(\lambda)$, then by Lemma \ref{kernelpointev} applied to $M_1$, there are functions $h_1,...,h_n \in M_1$ such that $f - g = \sum_{i=1}^n (\lambda_i - M_{z_i})h_i$. But then also
\begin{displaymath}
U(f-g) = \sum_{i=1}^n (\lambda_i - M_{z_i})Uh_i.
\end{displaymath}
Hence we obtain a well-defined surjective linear map $U_{\lambda} : M_{\lambda} \rightarrow (UM)_{\lambda}$ by setting
\begin{displaymath}
U_{\lambda} x = (Uf)(\lambda) \text{ if } f \in M \text{ with } f(\lambda) = x.
\end{displaymath}
It follows that $\text{fd}(M) = \dim M_{\lambda} \geq \dim (UM)_{\lambda} = \text{fd} (UM)$. By applying the same argument to $U^{-1}$ and $UM$ instead of $U$ and $M$ we find that also $\text{fd}(UM)\geq \text{fd}(M)$.
\end{proof}

If $\rho_i: X \rightarrow \mathcal{O}(\Omega_i, D_i)$ ($i=1,2$) are CF-representations on domains $\Omega_i \subseteq \Omega$ with non-trival intersection $\Omega_1 \cap \Omega_2 \neq \emptyset$, then the submodules $M_i = \rho_i X \subseteq \mathcal{O}(\Omega_i, D_i)$ are canonically isomorphic
\begin{displaymath}
M_1 \cong X / \ker \rho_1 = X / \ker \rho_2 \cong M_2
\end{displaymath}
as $\C[z]$-modules. As an application of the previous result one obtains that 
\begin{displaymath}
{\rm fd}(\rho_1 Y) = {\rm fd}(\rho_2 Y)
\end{displaymath}
for each linear subspace $Y \subseteq X$ which is invariant for $T$.

\begin{theorem}
\label{Invariance}
Let $\rho_i : X \rightarrow \mathcal{O}(\Omega_i, D_i)$ $(i=1,2)$ be CF-representations of $T$ on domains $\Omega_i \subseteq \Omega$. Then 
\begin{displaymath}
{\rm fd}(\rho_1 Y) = {\rm fd}(\rho_2 Y)
\end{displaymath}
for each linear subspace $Y \subseteq X$ which is invariant for $T$.
\end{theorem}

\begin{proof}
Since $\Omega$ is connected, we can choose a continuous path $\gamma : [0,1] \rightarrow \Omega$ such that $\gamma (0) \in  \Omega_1$ and $\gamma (1) \in \Omega_2$. By Theorem \ref{CFexists} there is a family $(\rho_z)_{z \in \text{Im}\gamma}$ of CF-representations $\rho_z: X \rightarrow \mathcal{O}(\Omega_z, D_z)$ of $T$ on connected open neighbourhoods $\Omega_z \subseteq \Omega$ of the points $z \in \text{Im}\gamma$ such that $\rho_{\gamma(0)} = \rho_1$ and $\rho_{\gamma(1)} = \rho_2$. Using the fact that there is a positive number $\delta > 0$ such that each set $A \subseteq [0,1]$ of diameter less than $\delta$ is completely contained in one of the sets $\gamma^{-1}(\Omega_z)$ (see e.g. Lemma 3.7.2 in \cite{Munkres}), one can choose a sequence of points $z_1 = \gamma(0),z_2,...,z_n = \gamma(1)$  in $\text{Im}\gamma$ such that $\Omega_{z_i} \cap \Omega_{z_{i+1}}\neq \emptyset$ for $i=1,...,n-1$. Let $Y \subseteq X$ be 
a linear $T$-invariant subspace. By the remarks following Lemma \ref{invariancenontrivial} we obtain that
\begin{displaymath}
\text{fd}(\rho_1 Y) = \text{fd}(\rho_{z_2} Y) = ...  = \text{fd} (\rho_2 Y)
\end{displaymath}
as was to be shown.
\end{proof}

Let $T \in L(X)^n$ be a weak dual Cowen-Douglas tuple of rank $N$ on a domain $\Omega \subseteq \C^n$ and let $Y \subseteq X$ be a linear subspace that is invariant for $T$. In view of Theorem \ref{Invariance} we can define the fiber dimension of $Y$ by 
\begin{displaymath}
\text{fd}(Y) = \text{fd}(\rho Y),
\end{displaymath}
where $\rho: X \rightarrow \mathcal{O}(\Omega_0, D)$ is an arbitrary CF-representation of $T$.
We shall mainly be interested in the fiber dimension of closed invariant subspaces $Y \in \text{Lat}(T)$, but the reader should observe that the definition makes perfect sense for linear $T$-invariant subspaces $Y \subseteq X$. Since by Theorem \ref{CFexists} there are always continuous CF-representations $\rho: X \rightarrow \mathcal{O}(\Omega_0, D)$ and since in this case the inclusions
\begin{displaymath}
\epsilon_{\lambda}(\rho(\overline{Y})) \subseteq \overline{\epsilon_{\lambda}(\rho(Y))} = \epsilon_{\lambda}(\rho(Y))
\end{displaymath}
hold for all $\lambda \in \Omega_0$, it follows that ${\rm fd}(Y) = {\rm fd}(\overline{Y})$ for each linear $T$-invariant subspace $Y \subseteq X$.

It follows from Theorem \ref{CFexists} that $\text{fd}(X) = N$. In general, the fiber dimension $\text{fd}(Y)$ of a linear $T$-invariant subspace $Y \subseteq X$ is an integer in $\{ 0,...,N \}$ which depends on $Y$ in a monotone way. Obviously, $\text{fd}(Y) = 0$ if and only if 
\begin{displaymath}
Y \subseteq \ker \rho = \bigcap_{z \in \Omega} (z-T)X^n.
\end{displaymath} 
We conclude this section with an alternative characterization of CF-repre- \\sentations.

\begin{corollary}
Let $T \in L(X)^n$ be a weak dual Cowen-Douglas tuple of rank $N$ on a domain $\Omega \subseteq \C^n$ and let $\rho: X \rightarrow \mathcal{O}(\Omega_0, D)$ be a $\C[z]$-module homomorphism on a domain $\emptyset \neq \Omega_0 \subseteq \Omega$ with a finite-dimensional vector space $D$ such that 
\begin{displaymath}
\ker \rho = \bigcap_{z \in \Omega} (z - T) X^n.
\end{displaymath}
Then $\rho$ is a CF-representation of $T$ if and only if ${\rm fd}(\rho X) = N$.
\end{corollary}

\begin{proof}
Suppose that $\text{fd}(\rho X ) = N$. Define $\hat{X} = \rho(X)$. Since the maps 
\begin{displaymath}
X / (\lambda - T) X^n \rightarrow \hat{X} / (\lambda - M_z) \hat{X}^n,\; [x] \mapsto [\rho x]
\end{displaymath}
and
\begin{displaymath}
\hat{X} / (\lambda - M_z) \hat{X}^n \rightarrow \hat{X}_{\lambda},\; [f] \mapsto f(\lambda) 
\end{displaymath}
are surjective for each $\lambda \in \Omega_0$, it follows that 
\begin{displaymath}
\dim \hat{X} / (\lambda - M_z) \hat{X}^n \leq N
\end{displaymath}
for all $\lambda \in \Omega_0$ and that equality holds on $\Omega_0 \backslash A$ with a suitable proper analytic subset $A \subseteq \Omega_0$. Equipped with the norm $\Vert \rho(x) \Vert = \Vert x + \ker \rho \Vert$, the space $\hat{X}$ is a Banach space and $M_z \in L(\hat{X})^n$ is a commuting tuple of bounded operators on $\hat{X}$. A result of Kaballo (Satz 1.5 in \cite{K}) shows that the set 
\begin{displaymath}
\{ \lambda \in \Omega_0; \dim \hat{X} / (\lambda - M_z) \hat{X}^n  > \min_{\mu \in \Omega_0} \dim \hat{X} / (\mu - M_z) \hat{X}^n \}
\end{displaymath}
is a proper analytic subset of $\Omega_0$. Combining these results we find that
\begin{displaymath}
\dim \hat{X} / (\lambda - M_z) \hat{X}^n  = N
\end{displaymath}
for all $\lambda \in \Omega_0$. Hence $\rho$ is a CF-representation of $T$.

Conversely, if $\rho$ is a CF-representation of $T$, then $\text{fd}(\rho X) = N$ by the remarks preceding the corollary. 
\end{proof}

\section{A limit formula for the fiber dimension}

Let $\Omega \subseteq \C^n$ be a domain with $0 \in \Omega$ and let $D$ be a finite-dimensional complex vector space. 
For $k \in \N$, let us consider the mapping $T_k : \mathcal{O}(\Omega, D) \rightarrow \mathcal{O}(\Omega, D)$ which 
associates with each function $f \in \mathcal{O}(\Omega, D)$ its $k$-th Taylor polynomial, that is,

\begin{displaymath}
T_k(f)(z) = \sum_{\vert \alpha \vert \leq k} \frac{f^{(\alpha)}(0)}{\alpha !} z^{\alpha}.
\end{displaymath}

In \cite{HSFT} (Lemma 1.4) it was shown that, for a given $\mathbb C[z]$-submodule,  there is a proper analytic subset $A \subseteq \Omega$ such that

\begin{displaymath}
\dim M_z = \max_{w \in \Omega} \dim M_w = n! \lim_{k \rightarrow \infty} \frac{\dim T_k(M)}{k^n}
\end{displaymath}

holds for all $z \in \Omega \backslash A$. \\
Based on this observation, we will deduce a similar limit formula for the fiber dimension of invariant subspaces of weak Cowen-Douglas tuples on $\Omega$.

Given a commuting tuple $T \in L(X)^n$ of bounded operators on a Banach space $X$, we write
\begin{align*}
K^{\bullet}(T, X): \xymatrix{0 \ar[r] & \Lambda^0(X) \ar[r]^{\delta_T^0} & \Lambda^1(X) \ar[r]^{\delta_T^1} & ... \ar[r]^{\delta_T^{n-1}} & \Lambda^n(X) \ar[r] & 0}
\end{align*}
for the Koszul complex of T (cf. Section 2.2 in \cite{EP}). For $i = 0,...,n$, let
\begin{displaymath}
H^i(T,X) = \ker(\delta_T^i) / {\rm Im}(\delta_T^{i-1})
\end{displaymath}
be the $i$-th cohomology group of $K^{\bullet}(T, X)$. There is a canonical isomorphism $H^n(T,X) \cong X / \sum_{i=1}^n T_iX$ of complex vector spaces.

In the following, given a commuting operator tuple $T \in L(X)^n$ and an invariant subspace $Y \in \text{Lat}(T)$, we denote by 
\begin{displaymath}
R = T\vert_Y \in L(Y)^n, S = T / Y \in L(Z)^n
\end{displaymath}
the restriction of $T$ to $Y$ and the quotient of $T$ modulo $Y$ on $Z = X / Y$. 
The inclusion $i: X \rightarrow Y$ and the quotient map $q: X \rightarrow Z$ induce a short exact sequence of complexes
\begin{align*}
\xymatrix{ 0 \ar[r] & K^{\bullet}(z - R, Y) \ar[r]^i & K^{\bullet}(z - T, X) \ar[r]^q & K^{\bullet}(z - S, Z) \ar[r] & 0}.
\end{align*}
It is a standard fact from homological algebra that there are connecting homomorphisms $d_z^i : H^i(z - S, Z) \rightarrow H^{i+1}(z - R, Y)$ ($i = 0,..., n-1$) such that the induced sequence of cohomology spaces

\begin{align*}
\entrymodifiers={+!!<0pt,\fontdimen22\textfont2>}
\xymatrix{ 0 \ar[r] & H^0(z - R, Y) \ar[r]^i & H^0(z - T, X) \ar[r]^q & H^0(z - S, Z) \\
			 \ar[r]^-{d_z^0} & H^1(z - R, Y) \ar[r]^i & H^1(z - T, X) \ar[r]^q & H^1(z - S, Z) \\
			 \ar[r]^-{d_z^1} & H^2(z - R, Y) \ar[r] & \quad \quad ... \quad \quad \\
			 \ar[r]^-{d_z^{n-1}} & H^n(z - R, Y) \ar[r]^i & H^n(z - T, X) \ar[r]^q & H^n(z - S, Z) \ar[r] & 0}
\end{align*}

is exact again. In particular, we obtain
\begin{align*}
\text{Im}(d_z^{n-1}) &= \ker (H^n(z - R, Y) \xrightarrow{i} H^n(z - T, X)) \\
&= (Y \cap (z-T)X^n) / (z - R)Y^n.
\end{align*}

\begin{lemma}

Let $T \in L(X)^n$ be a weak dual Cowen-Douglas tuple of rank $N$ on a domain $\Omega \subseteq \C^n$ and let $Y \in {\rm Lat}(T)$ be a closed invariant subspace of $T$. Then there is a proper analytic subset $A \subseteq \Omega$ such that 
\begin{displaymath}
\dim H^n (\lambda - S, Z) = N - {\rm fd}(Y)
\end{displaymath}
for all $\lambda \in \Omega \backslash A$.

\end{lemma}

\begin{proof}
Choose a CF-representation $\rho: X \rightarrow \mathcal{O}(\Omega_0, D)$ of $T$ on some domain $\Omega_0 \subseteq \Omega$ as in Theorem \ref{CFexists}. Let $Y \in \text{Lat}(T)$ be arbitrary. Define $\hat{X} = \rho(X)$ and $\hat{Y} = \rho(Y)$. Since the compositions 
\begin{displaymath}
Y^n \xrightarrow{\lambda - R} Y \xrightarrow{\rho} \mathcal{O}(\Omega_0, D) \xrightarrow{\epsilon_{\lambda}} D \quad (\lambda \in D)
\end{displaymath}
are zero, we obtain well-defined surjective linear maps
\begin{displaymath}
\delta_{\lambda}: H^n(\lambda - R, Y) \rightarrow \hat{Y}_{\lambda},\;  [y] \mapsto \rho(y)(\lambda). 
\end{displaymath}
Obviously, for each $\lambda \in \Omega$, the inclusion 
\begin{displaymath}
\text{Im} d_{\lambda}^{n-1} = (Y \cap (\lambda - T)X^n) / (\lambda - R)Y^n \subseteq  \ker \delta_{\lambda}
\end{displaymath}
holds. To see that also the reverse inclusion holds, fix an element $y \in Y$ with $\rho(y)(\lambda) = 0$. Since $\hat{X}$ is a divisible holomorphic model space, there are vectors $x_1,...,x_n  \in X$ with
\begin{displaymath}
\rho(y) = \sum_{i=1}^n (\lambda_i - M_{z_i}) \rho(x_i) = \rho(\sum_{i=1}^n (\lambda_i - T_i) x_i).
\end{displaymath}
But then 
\begin{displaymath}
y - \sum_{i=1}^n (\lambda_i - T_i) x_i \in \bigcap_{z \in \Omega} (z-T)X^n
\end{displaymath}
and hence $y \in Y \cap (\lambda - T)X^n$. Thus, for each $\lambda \in \Omega$, we obtain an exact sequence 
\begin{displaymath}
H^{n-1}(\lambda - S, Z) \xrightarrow{d_{\lambda}^{n-1}} H^n(\lambda - R, Y) \xrightarrow{\delta_{\lambda}} \hat{Y}_{\lambda} \rightarrow 0.
\end{displaymath}
Using the exactness of these sequences and of the long exact cohomology sequences explained in the section leading to Lemma 2.1, we find that
\begin{align*}
& \dim H^n(\lambda - S, Z) \\ &= \dim H^n (\lambda - T, X) - \dim H^n (\lambda - R, Y) / d_{\lambda}^{n-1} H^{n-1}(\lambda - S, Z)\\
&= N - \dim \hat{Y}_{\lambda}
\end{align*}
for all $\lambda \in \Omega$. Hence the assertion follows.
\end{proof}

By the cited result of Kaballo (Satz 1.5 in \cite{K}), in the setting of Lemma 2.1, the set
\begin{displaymath}
\{ \lambda \in \Omega; \dim H^n (\lambda - S, Z) > \min_{\mu \in \Omega} \dim H^n(\mu - S, Z) \}
\end{displaymath}
is an analytic subset of $\Omega$. It is well known that the minimum occurring here can be interpreted as a suitable Samuel multiplicity of the tuples $S - \mu$ for $\mu \in \Omega$.
Let us recall the necessary details.

For simplicity, we only consider the case where $\Omega$ is a domain in $\C^n$ with $0 \in \C^n$. For an arbitrary tuple $T \in L(X)^n$ of bounded operators on a Banach space $X$ with
\begin{displaymath}
\dim H^n (T,X) < \infty,
\end{displaymath}
all the spaces $M_k(T) = \sum_{\vert \alpha \vert = k} T^{\alpha} X$ ($k \in \N$) are finite codimensional in $X$ and the limit
\begin{displaymath}
c(T) = n! \lim_{k \rightarrow \infty} \frac{\dim X / M_k(T)}{k^n}
\end{displaymath}
exists. This number is referred to as the Samuel multiplicity of $T$. For each domain $\Omega \subseteq \C^n$ with $0 \in \Omega$ and $\dim H^n(\lambda - T, X) < \infty$ for all $\lambda \in \Omega$, there is a proper analytic subset $A \subseteq \Omega$ such that
\begin{displaymath}
c(T) = \dim H^n (\lambda - T, X) < \dim H^n (\mu - T, X)
\end{displaymath}
for all $\lambda \in \Omega \backslash A$ and $\mu \in A$ (see Corollary 3.6 in \cite{SMCO}). 
In particular, if $S \in L(Z)^n$ is as in Lemma 2.1 and $0 \in \Omega$, then the formula 
\begin{displaymath}
c(S) = N - \text{fd}(Y)
\end{displaymath}
holds. Hence
the following result from \cite{HSFT} allows us to deduce the announced limit formula for the fiber dimension.

\begin{lemma}
\label{HSFT} {\rm (Lemma 1.6 in \cite{HSFT})}
Let $T \in L(X)^n$ be a commuting tuple of bounded operators on a Banach space $X$, let $Y \in {\rm Lat}(T)$ be a closed invariant subspace and let $S = T/Y \in L(Z)^n$ be the induced quotient tuple on $Z = X/Y$. Suppose that 
\begin{displaymath}
\dim H^n(T,X) < \infty.
\end{displaymath}
Then the Samuel multiplicities of $T$ and $S$ satisfy the relation
\begin{displaymath}
c(S) = c(T) - n! \lim_{k \rightarrow \infty} \frac{\dim (Y + M_k(T))/M_k(T)}{k^n}.
\end{displaymath}
\end{lemma}

As a direct application we obtain a corresponding formula for the fiber dimension.

\begin{corollary}
\label{limitformula}
Let $T \in L(X)^n$ be a weak dual Cowen-Douglas tuple of rank $N$ on a domain $\Omega \subseteq \C^n$ with $0 \in \Omega$, and let $Y \in {\rm Lat}(T)$ be a closed invariant subspace for $T$. Then the formula
\begin{displaymath}
{\rm fd}(Y) = n! \lim_{k \rightarrow \infty} \frac{\dim (Y + M_k(T))/M_k(T)}{k^n}
\end{displaymath}
holds.
\end{corollary}

\begin{proof}
It suffices to observe that in the setting of Corollary \ref{limitformula} the identity $c(T) = N$ holds and then to compare the formula from Lemma \ref{HSFT} with the formula
\begin{displaymath}
c(S) = N - \text{fd}(Y)
\end{displaymath}
deduced in the section leading to Lemma \ref{HSFT}.
\end{proof}

For weak dual Cowen-Douglas tuples $T \in L(X)^n$ on general domains $\Omega \subseteq \C^n$ (not necessarily containing $0$), the above formula for ${\rm fd}(Y)$ remains true if on the right-hand side the spaces $M_k(T)$ are replaced by the spaces $M_k(T - \lambda_0)$ with $\lambda_0 \in \Omega$ arbitrary. This follows by an elementary translation argument.\\

If in Corollary \ref{limitformula} the space $X$ is a Hilbert space and if we write $P_k$ for the orthogonal projections onto the subspaces $M_k(T)^{\bot}$, then there are canonical vector space isomorphisms
\begin{displaymath}
(Y + M_k(T))/M_k(T) \rightarrow P_k Y,\;  [y] \mapsto P_k Y.
\end{displaymath}
Thus the resulting formula
\begin{displaymath}
\text{fd}(Y) = n! \lim_{k \rightarrow \infty} \frac{\dim(P_k Y)}{k^n}
\end{displaymath}
extends Theorem 19 in \cite{CCF}.

In the final result of this section we show that the fiber dimension ${\rm fd}(Y)$ is invariant under sufficiently small changes of the space $Y$. For given invariant subspaces $Y_1, Y_2 \in {\rm Lat}(T)$ with $Y_1 \subseteq Y_2$, we write $\sigma(T, Y_2 / Y_1)$ for the Taylor spectrum of the quotient tuple induced by $T$ on $Y_2 / Y_1$. 

\begin{corollary}
Let $T \in L(X)^n$ be a weak dual Cowen-Douglas tuple of rank $N$ on a domain $\Omega \subseteq \C^n$. Suppose that $Y_1, Y_2 \in {\rm Lat}(T)$ are closed $T$-invariant subspaces with $Y_1 \subseteq Y_2$ and $\Omega \cap (\C^n \backslash \sigma(T, Y_2 / Y_1)) \neq \emptyset$. Then ${\rm fd}(Y_1) = {\rm fd}(Y_2)$.
\end{corollary}

\begin{proof}
By Lemma 2.1 there is a point $\lambda \in \Omega \cap (\C^n \backslash \sigma(T, Y_1/Y_2))$ with 
\begin{displaymath}
\dim H^n (\lambda - T/Y_i, X/Y_i) = N - {\rm fd}(Y_i)
\end{displaymath}
for $i=1,2$. Using the long exact cohomology sequences induced by the canonical exact sequence
\begin{displaymath}
0 \rightarrow Y_2 / Y_1 \rightarrow Y / Y_1 \rightarrow Y / Y_2 \rightarrow 0
\end{displaymath}
one finds that the $n$-th cohomology spaces of $\lambda - T/Y_1$ and $\lambda - T/Y_2$ are isomorphic. Hence we obtain that ${\rm fd}(Y_1) = {\rm fd}(Y_2)$.
\end{proof}

To make the above proof work, it suffices that there is a point in $ \Omega$ which is not contained in the right spectrum of the quotient tuple induced by $T$ on $Y_2 / Y_1$ (cf. Section 2.6 in \cite{EP}). The hypotheses of Corollory 2.4 are satisfied for instance if $\dim (Y_2 / Y_1) < \infty$. Thus Corollary 2.4 can be seen as an extension of Proposition 2.5 in \cite{CGW}.

\section{Analytic Samuel multiplicity}
We briefly indicate an alternative way to calculate fiber dimensions which extends a corresponding idea from \cite{CCF}. Let $T \in L(X)^n$ be a commuting tuple of bounded operators on a Banach space $X$ and let $\Omega \subseteq \C^n$ be a domain such that
\begin{displaymath}
\dim H^n(\lambda - T, X) < \infty
\end{displaymath}
for all $\lambda \in \Omega$. For simplicity, we again assume that $0 \in \Omega$. By Corollary 2.2 in \cite{SMCO} the quotient sheaf 
\begin{displaymath}
\mathcal{H}_T = \mathcal{O}_{\Omega}^X / (z - T) \mathcal{O}_{\Omega}^{X^n}
\end{displaymath}
of the sheaf of all analytic $X$-valued functions on $\Omega$ is a coherent analytic sheaf on $\Omega$. Let $Y \in \text{Lat}(T)$ be a closed invariant subspace for $T$. As before denote by $R = T\vert_Y \in L(Y)^n$ the restriction of $T$ and by $S = T/Y \in L(Z)^n$ the quotient tuple induced by $T$ on $Z = X / Y$. Let $i: Y \rightarrow X$ and $q: X \rightarrow Z$ be the inclusion and quotient map, respectively. Then
\begin{displaymath}
0 \rightarrow K^{\bullet}(z- R, \mathcal{O}_{\Omega}^Y) \xrightarrow{i} K^{\bullet}(z-T, \mathcal{O}_{\Omega}^X) \xrightarrow{q} K^{\bullet}(z - S, \mathcal{O}_{\Omega}^Z) \rightarrow 0
\end{displaymath}
is a short exact sequence of complexes of analytic sheaves on $\Omega$. Passing to stalks and using the induced long exact cohomology sequences, one finds that the upper horizontal in the commutative diagram
\begin{align*}
\xymatrix{&\mathcal{H}_R \ar[r]^i & \mathcal{H}_T \ar[r]^q &\mathcal{H}_S \rightarrow 0 \\
			& \mathcal{O}_{\Omega}^Y \ar@{->>}[u]^{\pi_Y}  \ar[r]^i & \mathcal{O}_{\Omega}^X \ar@{->>}[u]_{\pi_X}}
\end{align*}
is an exact sequence of analytic sheaves. Here $\pi_Y$ and $\pi_X$ denote the canonical quotient maps. The sheaf $\mathcal{M} = \pi_X(i \mathcal{O}_{\Omega}^Y)$ is the kernel of the surjective sheaf homomorphism
\begin{displaymath}
\mathcal{H}_T \xrightarrow{q} \mathcal{H}_S.
\end{displaymath}
Since $\mathcal{H}_T$ and $\mathcal{H}_S$ are coherent, also the sheaf $\mathcal{M}$ is a coherent analytic sheaf on $\Omega$ (Satz 26.13 in \cite{Kultze}). Hence
\begin{displaymath}
0 \rightarrow \mathcal{M}_0 \xrightarrow{i} \mathcal{H}_{T, 0} \xrightarrow{q} \mathcal{H}_{S, 0} \rightarrow 0
\end{displaymath}
is an exact sequence of Noetherian $\mathcal{O}_0$-modules. For a Noetherian $\mathcal{O}_0$-module $E$, let us denote by $e_{\mathcal{O}_0}(E)$ its analytic Samuel multiplicity, that is, the multiplicity of $E$ with respect to the multiplicity system $(z_1,...,z_n)$ on $E$ (see Section 7.4 in \cite{Northcott}). Since the analytic Samuel multiplicity is additive with respect to short exact sequences of Noetherian $\mathcal{O}_0$-modules (Theorem 7.5 in \cite{Northcott}), it follows that 
\begin{displaymath}
e_{\mathcal{O}_0}(\mathcal{H}_{T,0}) = e_{\mathcal{O}_0}(\mathcal{M}_0) + e_{\mathcal{O}_0}(\mathcal{H}_{S,0}).
\end{displaymath}
By Corollary 4.1 in \cite{SMCO} the analytic Samuel multiplicities $e_{\mathcal{O}_0}(\mathcal{H}_{T,0})$ and $e_{\mathcal{O}_0}(\mathcal{H}_{S,0})$ coincide with the Samuel multiplicities $c(T)$ and $c(S)$ as defined in Section 2. Thus we obtain the identity
\begin{displaymath}
c(T) = e_{\mathcal{O}_0}(\mathcal{M}_0) + c(S).
\end{displaymath}
By Theorem 8.5 in \cite{Northcott} the analytic Samuel multiplicity $e_{\mathcal{O}_0}(\mathcal{M}_0)$ can also be calculated as the Euler characteristic $\chi(K^{\bullet}(z, \mathcal{M}_0))$ of the Koszul complex of the multiplication operators with $z_1,...,z_n$ on $\mathcal{M}_0$. \\
Summarizing we obtain the following result.
\begin{theorem}
Let $T \in L(X)^n$ be a weak dual Cowen-Douglas tuple on a domain $\Omega \subseteq \C^n$ with $0 \in \Omega$ and let $Y \in {\rm Lat}(T)$ be a closed invariant subspace for $T$. Then with the notation from above, the fiber dimension of $Y$ can be calculated as 
\begin{displaymath}
{\rm fd}(Y) = n! \lim_{k \rightarrow \infty} \frac{\dim (Y + M_k(T))/M_k(T)}{k^n} = e_{\mathcal{O}_0}(\mathcal{M}_0).
\end{displaymath}
\end{theorem}

\section{A lattice formula for the fiber dimension}
Let $T \in L(X)^n$ be a weak dual Cowen-Douglas tuple of rank $N$ on a domain $\Omega \subseteq \C^n$ and let $Y_1,Y_2 \in \text{Lat}(T)$ be closed invariant subspaces. A natural problem studied in \cite{CCF} is to find conditions under which the dimension formula
\begin{displaymath}
\text{fd}(Y_1) + \text{fd}(Y_2) = \text{fd}(Y_1 \vee Y_2) + \text{fd}(Y_1 \cap Y_2)
\end{displaymath}
holds. Note that, for a dual Cowen-Douglas tuple of rank $1$, the validity of this formula for all closed invariant subspaces $Y_1, Y_2$ is equivalent to the condition that any two non-zero closed invariant subspaces $Y_1,Y_2$ have a non-trivial intersection. As observed in \cite{CCF} elementary linear algebra can be used to obtain at least an inequality.

\begin{lemma}
\label{latticegreater}
Let $T \in L(X)^n$ be a weak dual Cowen-Douglas tuple on a domain $\Omega \subseteq \C^n$ and let $Y_1,Y_2 \subseteq X$ be linear $T$-invariant subspaces. Then the inequality 
\begin{displaymath}
{\rm fd}(Y_1) + {\rm fd}(Y_2) \geq {\rm fd}(Y_1 + Y_2) + {\rm fd}(Y_1 \cap Y_2)
\end{displaymath}
holds.
\end{lemma}

\begin{proof}
Let $\rho: X \rightarrow \mathcal{O}(\Omega_0, D)$ be a CF-representation of $T$ on a domain $\Omega_0 \subseteq \Omega$. It suffices to observe that, for each point $\lambda \in \Omega_0$, the estimate
\begin{align*}
\dim \epsilon_{\lambda} \rho(Y_1 + Y_1) &= \dim \epsilon_{\lambda} \rho(Y_1) + \dim \epsilon_{\lambda} \rho(Y_2) - \dim (\epsilon_{\lambda} \rho(Y_1) \cap \epsilon_{\lambda} \rho(Y_2)) \\
&\leq \dim \epsilon_{\lambda} \rho(Y_1) + \dim \epsilon_{\lambda} \rho(Y_2) - \dim \epsilon_{\lambda} \rho(Y_1 \cap Y_2)
\end{align*}
holds and then to choose $\lambda$ as a common maximal point for the submodules $ \rho(Y_1 + Y_2)$, $\rho(Y_1)$, $\rho(Y_2)$ and $ \rho(Y_1 \cap Y_2)$. 
\end{proof}

Note that, for closed invariant subspaces $Y_1, Y_2 \in \text{Lat}(T)$, the inequality in \ref{latticegreater} can be rewritten as
\begin{displaymath}
\text{fd}(Y_1) + \text{fd}(Y_2) \geq \text{fd}(Y_1 \vee Y_2) + \text{fd}(Y_1 \cap Y_2).
\end{displaymath}

Let $\Omega \subseteq \C^n$ be a domain and let $D$ be an $N$-dimensional complex vector space. We shall say that a function $f \in \mathcal{O}(\Omega, D)$ has coefficients in a given subalgebra $A \subseteq \mathcal{O}(\Omega)$ if the coordinate functions of $f$ with respect to some, or equivalently, every basis of $D$ belong to $A$. Let $M \subseteq \mathcal{O}(\Omega, D)$ be a $\C[z]$-submodule. We shall say that $A$ is dense in $M$ if every function $f \in M$ is the pointwise limit of a sequence $(f_k)_{k \in \N}$ of functions in $M$ such that each $f_k$ has coordinate functions in $A$.

\begin{theorem} 
\label{generalequality}
Let $M_1, M_2 \subseteq \mathcal{O}(\Omega, D)$ be $\C[z]$-submodules such that $A$ is dense in $M_1$ and in $M_2$ and such that $AM_i \subseteq M_i$ for $i=1,2$. Then we have
\begin{displaymath}
{\rm fd}(M_1 + M_2) + {\rm fd}(M_1 \cap M_2) = {\rm fd}(M_1) + {\rm fd}(M_2).
\end{displaymath}
\end{theorem}

\begin{proof}
Exactly as in the proof of Lemma 4.1 it follows that
\begin{displaymath}
{\rm fd}(M_1 + M_2) + {\rm fd}(M_1 \cap M_2) \leq {\rm fd}(M_1) + {\rm fd} (M_2).
\end{displaymath}
To prove the reverse inequality it suffices to check that the arguments used in \cite{CF} to prove the corresponding result for invariant subspaces of analytic functional Hilbert spaces $H(K)$ given by a complete Nevanlinna-Pick kernel on a domain in $\C$ remain valid. For the convenience of the reader, we indicate the main ideas. \\
Define $M = M_1 + M_2$ and choose a point $\lambda \in \Omega$ which is maximal with respect to $M_1$, $M_2$ and $M$. Define $E = (M_1)_{\lambda} \cap (M_2)_{\lambda}$ and choose direct complements $E_1$ of $E$ in $(M_1)_{\lambda}$ and $E_2$ of $E$ in $(M_2)_{\lambda}$. Fix bases $(e_1,...,e_{d_1})$ of $E_1$, $(e_{d_1 + 1},..., e_{d_1 + d_2})$ for $E_2$ and $(e_{d_1 + d_2 + 1},...,d_{d_1 + d_2 + d'})$ for $E$, where $d_1,d_2,d' \geq 0$ are non-negative integers. Set $d = d_1 + d_2 + d'$. An elementary argument shows that $(e_1,...,e_d)$ is a basis of $M_{\lambda}$. Let us complete this basis to a basis $B = (e_1,...,e_d,e_{d+1},...,e_N)$ of $D$. Since $\text{fd}(M_1) + \text{fd}(M_2) - \text{fd}(M) = d'$, we have to show that 
\begin{displaymath}
\text{fd}(M_1 \cap M_2) \geq d'.
\end{displaymath}
We may of course assume that $d' \neq 0$. Since $A$ is dense in $M$, there are functions $h_1,...,h_d \in M$ with
\begin{displaymath}
h_i(\lambda) = e_i \quad \quad (i=1,...,d)
\end{displaymath}
such that each $h_i$ has coefficients in $A$. Write
\begin{displaymath}
h_i = \sum_{j=1}^N h_{ij} e_j \quad \quad (i=1,...,d).
\end{displaymath}
Then $\theta = (h_{ij})_{1 \leq i,j \leq d}$ is a $(d \times d)$-matrix with entries in $A$ such that $\theta(\lambda) = E_d$ is the unit matrix. By basic linear algebra there is a $(d \times d)$-matrix $(A_{ij})$ with entries in $A$ such that $(A_{ij}) \theta = \text{diag}(\det \theta)$ is the $(d \times d)$-diagonal matrix with all diagonal terms equal to $\det(\theta)$. Then
\begin{displaymath}
(A_{ij})_{1 \leq i,j \leq d} (h_{ij})_{\substack{1 \leq i \leq d \\ 1 \leq j \leq N}} = (\text{diag}(\det \theta), (g_{ij})),
\end{displaymath}
where $(g_{ij})$ is a suitable matrix with entries in $A$. We define functions $H_1,...,H_d \in M$ by setting
\begin{displaymath}
H_i = \det (\theta) e_i + \sum_{j=1}^{N-d} g_{ij} e_{d+j} = \sum_{j=1}^N (\sum_{\nu = 1}^d A_{i \nu} h_{\nu j}) e_j = \sum_{\nu = 1}^d A_{i \nu} h_{\nu}.
\end{displaymath}
By construction $H_i(\lambda) = e_i$ and $(H_1(z),...,H_d(z))$ is a basis of $M_z$ for every point $z \in \Omega$ with $\det (\theta(z)) \neq 0$. 
If $f = f_1e_1+....+f_N e_N \in M$ is arbitrary, then at each point $z \in \Omega$  which is not contained in the zero set $Z(\det (\theta))$ of the analytic function $\det (\theta) \in \mathcal{O}(\Omega)$, the function $f$ can be written as a linear combination
\begin{displaymath}
f(z) = \lambda_1(z, f) H_1(z) + ... + \lambda_d(z, f) H_d(z).
\end{displaymath}
Using the definition of the functions $H_i$, we find that
\begin{displaymath}
f_1 = \lambda_1 (\cdot, f) \det ( \theta),..., f_d = \lambda_d(\cdot, f) \det(\theta).
\end{displaymath}
Hence, for $j = d+1,...,N$ and $z \in \Omega \backslash Z(\det \theta)$, we obtain that
\begin{align*}
f_j(z) &= \lambda_1(z,f) g_{1,j-d}(z) + ... + \lambda_d(z,f) g_{d, j-d}(z) \\
&= \frac{g_{1, j-d}(z)}{\det \theta(z)} f_1(z) + .... + \frac{g_{d,j-d}(z)}{\det \theta(z)} f_d(z).
\end{align*}
In particular, each function $f = f_1 e_1 + ... + f_N e_N \in M$ is uniquely determined by its first $d$ coordinate functions $(f_1,...,f_d)$. \\

Since $A$ is dense in $M_1$ and in $M_2$, we can choose functions $F_1,...,F_{d_1 + d'} \in M_1$ and $G_1,...,G_{d_2 + d'} \in M_2$ with coefficients in $A$ such that
\begin{displaymath}
(F_i(\lambda))_{i=1,...,d_1+d'} = (e_1,...,e_{d_1},e_{d_1 + d_2 + 1},...,e_{d_1 + d_2 + d'})
\end{displaymath}
and
\begin{displaymath}
(G_i(\lambda))_{i=1,...,d_2+d'} = (e_{d_1+1},...,e_{d_1 + d_2 + d'}).
\end{displaymath}
Write the first $d$ coordinate functions of each of the functions $$F_1,...,F_{d_1},G_1,...,G_{d_2}, F_{d_1+1},...,F_{d_1+d'},G_{d_2+1},...,G_{d_2+d'}$$ with respect to the basis $(e_1,...,e_N)$ of $D$ as column vectors and arrange these column vectors to a matrix $\Delta$ in the indicated order. Then $\Delta$ is a $(d \times (d+d'))$-matrix with entries in $A$. Write $\Delta = (\Delta_0, \Delta_1)$ where $\Delta_0$ is the $(d \times d)$-matrix consisting of the first $d$ columns of $\Delta$ and $\Delta_1$ is the $(d \times d')$-matrix consisting of the last $d'$ columns of $\Delta$. \\
By construction we have $\det (\Delta_0(\lambda)) = 1$. On $\Omega \backslash Z(\det \Delta_0)$, we can write
\begin{displaymath}
(\det \Delta_0) \Delta_0^{-1} \Delta = (\text{diag}(\det \Delta_0), \Gamma),
\end{displaymath}
where $\text{diag}(\det \Delta_0)$ is the $(d \times d)$-diagonal matrix with all diagonal terms equal to $\det \Delta_0$ and $\Gamma = (\gamma_{ij})$ is a $(d \times d')$-matrix with entries in $A$. The column vectors
\begin{displaymath}
r_j = (\gamma_{1j},...,\gamma_{dj},0,...,0, - \det \Delta_0,0,...,0)^t \quad (j = 1,...,d'),
\end{displaymath}
where $- \det \Delta_0$ is the entry in the $(d+j)$-th position, satisfy the equations
\begin{displaymath}
(\det \Delta_0) \Delta_0^{-1} \Delta r_j = ((\det \Delta_0) \gamma_{ij} - (\det \Delta_0) \gamma_{ij})_{i=1}^d = 0
\end{displaymath}
on $\Omega \backslash Z(\det \Delta_0)$. Hence $\Delta r_j = 0$ for $j = 1,...,d'$, or equivalently, for each $j=1,...,d$, the first $d$ coordinate functions of
\begin{displaymath}
\gamma_{1j} F_1 + ... + \gamma_{d_1 j} F_{d_1} + \gamma_{d_1 + d_2 + 1, j} F_{d_1 + 1} + ... + \gamma_{d_1 + d_2 + d', j} F_{d_1 + d'}
\end{displaymath}
with respect to $(e_1,...,e_N)$ coincide with those of 
\begin{displaymath}
(\det \Delta_0) G_{d_2 + j} - \gamma_{d_1 + 1, j} G_1 - ... - \gamma_{d_1 + d_2, j} G_{d_2}.
\end{displaymath}
Since, for each $j$, both functions belong to $M$, they coincide. But then these functions belong to $M_1 \cap M_2$. Since the vectors
\begin{displaymath}
G_i(\lambda) = e_{d_1 + i} \quad \quad (i = 1,...,d_2+ d')
\end{displaymath}
are linearly independent and since $\det (\Delta_0(\lambda)) = 1$, it follows that ${\rm fd}(M_1 \cap M_2) = \dim(M_1 \cap M_2)_{ \lambda} \geq d'$.
\end{proof}

Suppose for the moment that $\Omega \subseteq \C^n$ is a Runge domain. Since by the Oka-Weil approximation theorem,
the polynomials are dense in $\mathcal{O}(\Omega)$ with respect to the Fr\'{e}chet space topology of uniform 
convergence on compact subsets, each $\C[z]$-submodule $M \subseteq \mathcal{O}(\Omega, D)$ which is closed with respect to the Fr\'{e}chet space topology of $\mathcal{O}(\Omega, D)$ is automatically an $\mathcal{O}(\Omega)$-submodule. Hence we obtain the following consequence of Theorem \ref{generalequality}.

\begin{corollary}
Let $\Omega \subseteq \C^n$ be a Runge domain and let $D$ be a finite-dimensional complex vector space. Then the fiber dimension formula
\begin{displaymath}
{\rm fd}(M_1 + M_2) + {\rm fd}(M_1 \cap M_2) = {\rm fd}(M_1) + {\rm fd}(M_2)
\end{displaymath}
holds for each pair of closed $\C[z]$-submodules $M_1, M_2$ of the Fr\'{e}chet space $\mathcal{O}(\Omega, D)$.
\end{corollary}

Suppose that $T \in L(X)^n$ is a dual Cowen-Douglas tuple of rank $N$ on a domain $\Omega \subseteq \C^n$. Choose a CF-representation
\begin{displaymath}
\rho: X \rightarrow \mathcal{O}(\Omega_0, D)
\end{displaymath}
of $T$ as in the proof of Theorem \ref{CFexists}. Let $M \in {\rm Lat}(T)$ be an invariant subspace of $T$ such that each vector $m \in M$ is the limit of a sequence of vectors in
\begin{displaymath}
M \cap {\rm span}\{ T^{\alpha} x; \alpha \in \N^n \text{ and } x \in D \}.
\end{displaymath}
Then $\rho(M) \subseteq \mathcal{O}(\Omega_0, D)$ is a $\C[z]$-submodule in which the polynomials are dense in the sense explained in the section leading to Theorem \ref{generalequality}. Hence, for any two invariant subspaces $M_1, M_2 \in {\rm Lat}(T)$ of this type, the fiber dimension formula 
\begin{align*}
{\rm fd}(M_1 + M_2) + & {\rm fd} (M_1 \cap M_2) = {\rm fd}(\rho(M_1) + \rho(M_2)) + {\rm fd}(\rho(M_1) \cap \rho(M_2)) \\
&= {\rm fd}(\rho(M_1)) + {\rm fd}(\rho(M_2)) = {\rm fd}(M_1) + {\rm fd}(M_2)
\end{align*}
holds. The above density condition on $M$ is trivially fulfilled for every closed $T$-invariant subspace $M$ which is generated by a subset of $D$. But there are other situations to which this observation applies.

Recall that a commuting tuple $T \in L(H)^n$ of bounded operators on a complex Hilbert space $H$ is called graded if $H = \bigoplus_{k=0}^{\infty} H_k$ is the orthogonal sum of closed subspaces $H_k \subseteq H$ such that $\dim H_0 < \infty$ and 
\begin{enumerate}[(i)]
\item $T_j H_k \subseteq H_{k+1}$ \quad \quad (k$ \geq 0, j = 1,...,n$),
\item $\sum_{j=1}^n T_j H \subseteq H$ is closed,
\item $\bigvee_{\alpha \in \N^n} T^{\alpha} H_0 = H$.
\end{enumerate}
It is elementary to show (Lemma 2.4 in \cite{Groth}) that under these hypotheses the identities 
\begin{displaymath}
\sum_{\vert \alpha \vert = k} T^{\alpha} H = \bigoplus_{j=k}^{\infty} H_j \text{ and } \sum_{\vert \alpha \vert = k} T^{\alpha} H_0 = H_k
\end{displaymath}
hold for all integers $k \geq 0$. By definition a closed invariant subspace $M \in {\rm Lat}(T)$ of a graded tuple $T \in L(H)^n$ is said to be homogeneous if
\begin{displaymath}
M = \bigoplus_{k=0}^{\infty} M \cap H_k.
\end{displaymath}

\begin{corollary}
\label{homogen}
Let $T \in L(H)^n$ be a graded dual Cowen-Douglas tuple on a domain $\Omega \subseteq \C^n$. Then the fiber dimension formula
\begin{displaymath}
{\rm fd}(M_1 + M_2) + {\rm fd}(M_1 \cap M_2) = {\rm fd}(M_1) + {\rm fd}(M_2)
\end{displaymath}
holds for any pair of homogeneous invariant subspaces $M_1, M_2 \in {\rm Lat}(T)$.
\end{corollary}

\begin{proof}
By the remarks preceding the corollary
\begin{displaymath}
H = (\sum_{j=1}^n T_jH) \oplus H_0.
\end{displaymath}
Hence in the proof of Theorem \ref{CFexists} we can choose $D = H_0$. Let $\rho: H \rightarrow \mathcal{O}(\Omega_0, H_0)$ be a CF-representation of $T$ as constructed in the proof of Theorem \ref{CFexists}. Let $M \in {\rm Lat}(T)$ be a homogeneous invariant subspace for $T$. Then each element $m \in M$ can be written as a sum $m = \sum_{k=0}^{\infty} m_k$ with
\begin{displaymath}
m_k \in M \cap \sum_{\vert \alpha \vert = k} T^{\alpha} H_0 \quad (k \in \N).
\end{displaymath}
Hence the assertion follows from the remarks preceding Corollary \ref{homogen}.
\end{proof}

Typical examples of graded dual Cowen-Douglas tuples are multiplication tuples $M_z = (M_{z_1},...,M_{z_n}) \in L(H)^n$ with the coordinate functions on analytic functional Hilbert spaces $H = H(K_f, \C^N)$ given by a reproducing kernel
\begin{displaymath}
K_f: B_r(a) \times B_r(a) \rightarrow L(\C^n), K_f(z,w) = f(\langle z, w \rangle) 1_{\C^N},
\end{displaymath}
where $f(z) = \sum_{n=0}^{\infty} a_n z^n$ is a one-variable power series with radius of convergence $R = r^2 > 0$ such that $a_0 = 1$, $a_n > 0$ for all $n$ and
\begin{displaymath}
0 < \inf_{n \in \N} \frac{a_n}{a_{n+1}} \leq \sup_{n \in \N} \frac{a_n}{a_{n+1}} < \infty
\end{displaymath}
(see \cite{GHX} or \cite{Wernet}). In this case $H$ is the orthogonal sum 
\begin{displaymath}
H = \bigoplus_{k=0}^{\infty} \mathbb{H}_k \otimes \C^N
\end{displaymath}
of the subspaces consisting of all homogeneous $\C^N$-valued polynomials of degree $k$ and every invariant subspace
\begin{displaymath}
M = \bigvee_{i=1}^r \C[z] p_i \in {\rm Lat}(M_z)
\end{displaymath}
generated by a finite set of homogeneous polynomials $p_i \in \mathbb{H}_{k_i} \otimes \C^N$ is homogeneous. This class of examples contains the Drury-Arveson space, the Hardy space and the weighted Bergman spaces on the unit ball. \\

Let $H = H(K) \subseteq \mathcal{O}(\Omega)$ be an analytic functional Hilbert space on a domain $\Omega \subseteq \C^n$, or equivalently, a functional Hilbert space given by an analytic reproducing kernel $K: \Omega \times \Omega \rightarrow \C$. Let $D$ be  a finite-dimensional complex Hilbert space. Then the $D$-valued functional Hilbert space $H(K_D) \subseteq \mathcal{O}(\Omega, D)$ given by the kernel
\begin{displaymath}
K_D: \Omega \times \Omega \rightarrow L(D), K_D(z,w) = K(z,w) 1_D
\end{displaymath}
can be identified with the Hilbert space tensor product $H(K) \otimes D$. Let us denote by $M(H) = \{ \varphi : \Omega \rightarrow \C; \varphi H \subseteq H \}$ the multiplier algebra of $H$.

\begin{corollary}

Suppose that $H = H(K)$ contains all constant functions and that $z_1,...,z_n \in M(H)$.

\begin{enumerate}[(a)]
\item For any pair of closed subspaces $M_1, M_2 \subseteq H(K_D)$ such that $M(H)M_i \subseteq M_i$ for $i=1,2$ and such that $M(H)$ is dense in $M_1$ and $M_2$, the fiber dimension formula
\begin{displaymath}
{\rm fd}(M_1 \vee M_2) + {\rm fd}(M_1 \cap M_2) = {\rm fd}(M_1) + {\rm fd}(M_2)
\end{displaymath}
holds.
\item If in addition $K$ is a complete Nevanlinna-Pick kernel, that is, $K$ has no zeros and also the mapping $1 - \frac{1}{K}$ is positive definite,
then the fiber dimension formula holds for all closed subspaces $M_1, M_2 \subseteq H(K_D)$ which are invariant for $M(H)$.
\end{enumerate}

\end{corollary}

\begin{proof}
Part (a) is a direct consequence of Theorem \ref{generalequality}. If $K$ is a complete Nevanlinna-Pick kernel, then the Beurling-Lax-Halmos theorem for Nevanlinna-Pick spaces proved by McCullough and Trent (see Theorem 8.67 in \cite{AMC} or Theorem 3.3.8 in \cite{B}) implies that $M(H)$ is dense in every closed subspace $M \subseteq H(K_D)$ which is invariant for $M(H)$.

\end{proof}

Note that the condition that $M(H)$ is dense in a subspace $M \subseteq H(K_D)$ is satisfied for every closed $M(H)$-invariant subspace $M\subseteq H(K_D)$ that is generated by an arbitrary family of functions $f_i: \Omega \rightarrow D$ ($i \in I$) with coefficients in $M(H)$. Part (b) for domains $\Omega \subseteq \C$ was proved in \cite{CCF}. The proof in the multivariable case is the same.

{\small J\"{o}rg Eschmeier\\
Fachrichtung Mathematik\\
Universit\"{a}t des Saarlandes \\
Postfach 151150\\ D-66041 Saarbr\"{u}cken, Germany \\
e-mail: eschmei@math.uni-sb.de}\\

\vspace{.2cm}
{\small Sebastian Langend\"{o}rfer\\
Fachrichtung Mathematik\\
Universit\"{a}t des Saarlandes\\
Postfach 151150\\ D-66041 Saarbr\"{u}cken, Germany \\
e-mail: langendo@math.uni-sb.de}

\end{document}